\documentclass[draft]{amsproc}
\usepackage{amsmath}
\usepackage{amsfonts}
\usepackage[dvips]{graphicx}
\usepackage{amssymb}
\usepackage{amsbsy}
\usepackage{amsthm}
\usepackage{graphics}
\usepackage{color}
\usepackage{textcomp}

\makeatletter
\@namedef{subjclassname@2010}{%
\textup{2010} Mathematics Subject Classification}
\makeatother

\DeclareMathOperator{\re}{Re}

\numberwithin{equation}{section}

\newtheorem{theorem}{Theorem}[section]

\newtheorem{lemma}[theorem]{Lemma}

\theoremstyle{remark}
\newtheorem{rem}[theorem]{Remark}

\newcommand{\hh}{\mathcal{H}}
\newcommand{\hil}{\mathcal{H}}

\newcommand{\natu}{\mathbb{N}}
\newcommand{\Mm}{\mathcal{M}}

\newcommand{\nul}{\mathcal{N}(}
\newcommand{\real}{\mathbb{R}}

\theoremstyle{definition}
\newtheorem{ex}[theorem]{Example}

\newcommand*{\Ge}{\geqslant}
\newcommand{\ran}{\mathcal{R}(}
\newcommand{\M}{{A^\prime}}
\begin{document}

\title{Reduced commutativity of moduli of operators}
   \author[P. Pietrzycki]{Pawe{\l} Pietrzycki}
   \subjclass[2010]{Primary 47B20; Secondary
47B37} \keywords{normal operator, quasinormal operator, operator convex function, Davis-Choi-Jensen  inequality, operator equation, operator inequality}
   \address{Wydzia{\l} Matematyki i Informatyki, Uniwersytet
Jagiello\'{n}ski, ul. {\L}ojasiewicza 6, PL-30348
Krak\'{o}w}
   \email{pawel.pietrzycki@im.uj.edu.pl}
   \begin{abstract}
In this paper, we investigate the question of when the equations  $A^{*s}A^{s}=(A^{*}A)^{s}$ for all $s \in S$, where $S$ is a finite set of positive integers, imply the quasinormality or normality of $A$. In particular, it is proved that if 
$S=\{p,m,m+p,n,n+p\}$, where $2\leq m < n$, then $A$ is quasinormal. Moreover, if $A$ is invertible and $S=\{m,n,n+m\}$, where $m \leq n$, then $A$ is normal. Furthermore, the case when $S=\{m,m+n\}$ and $A^{*n}A^n \leq (A^*A)^n$ is discussed. 
   \end{abstract}
   \maketitle
   \section{Introduction}
The class of bounded quasinormal operators was
introduced by A. Brown in \cite{brow}. A bounded operator $A$ on a (complex)  Hilbert space $\hh$ is said to be \textit{quasinormal} if $A(A^*A)=(A^*A)A$. Two different
definitions of unbounded quasinormal operators
appeared independently in \cite{kauf} and in
\cite{szaf}. As recently shown in \cite{jabl}, these
two definitions are equivalent. Following \cite{szaf},
we say that a closed densely defined operator $A$ in  $\mathcal{H}$ is {\em
quasinormal} if $A$ commutes with the spectral measure
$E$ of $|A|$, i.e $E(\sigma) A \subset AE(\sigma)$ for
all Borel subsets $\sigma$ of the nonnegative part of
the real line. By \cite[Proposition 1]{szaf}, a closed
densely defined operator $A$ in $\mathcal{H}$ is
quasinormal if and only if $U|A|\subset |A|U$, where
$A = U|A|$ is the polar decomposition of $A$ (cf.\
\cite[Theorem 7.20]{weid}). For more
information on quasinormal operators we refer the
reader to \cite{brow,conw,uch}, the bounded case, and to
\cite{kauf,szaf,maj,jabl,qext,uch}, the unbounded one.

In 1973 M. R. Embry published a very influential paper \cite{embry} concerning the Halmos-Bram criterion for subnormality. In particular, she  gave a characterization of the class of quasinormal operators in terms of powers of operators. Namely, a bounded operator $A$ in a Hilbert space is quasinormal  if and only if the following condition holds 
\begin{equation}\label{wst}
    A^{*n}A^{n}=(A^*A)^n\qquad \text{for all}\quad n\in \natu.
\end{equation}
This leads to the
following question:
 is it necessary to assume that the equality in \eqref{wst} holds for all $n\in\natu$? To be more
precise we ask  for which subset $S\subset \mathbb{N}$  the following system of operator equations: 
\begin{equation}\label{xukl}
 A^{*s}A^s=(A^*A)^s\qquad\text{for all}\quad s \in S,
\end{equation}
implies the quasinormality of $A$.

It is worth mentioning that in some sense similar problem was studied and solved in Group Theory. We say that a group $G$ is \textit{ $n$-Abelian} if $( xy)^n = x^ny^n$ for all $x, y\in G$.
We call a set of integers $S$ \textit{Abelian forcing} if whenever $G$ is a
group with the property that $G$ is $n$-Abelian for all $n$ in $S$, then $G$ is Abelian. Then we have the following theorem
\begin{theorem} $($cf.
\cite{gal}$)$ A set $S$ of integer is Abelian forcinig if and only if the greatest common
divisor of the integer $n( n - 1)$ as $n$ ranges over $S$ is $2$. 

\end{theorem}

In  paper \cite{uch}
M. Uchiyama proved that if bounded  operator $A$ in Hilbert space is compact (in particular, is in finite dimensional space) or subnormal then the single equality 
\begin{equation}\label{xrow}
    A^{*n}A^n=(A^*A)^n
\end{equation}
for $n\geq 2$ implies quasinormality of $A$. He also proved that, if one of the following conditions holds 
\begin{itemize}
\item[(i)] $A$ is a hyponormal operator and satisfies  \eqref{xukl} with $S=\{n,n+1\}$, where $n\in \natu$,
\item[(ii)] $A$ is an operator in separable Hilbert space and satisfies  \eqref{xukl} with $S=\{k,k+1,l,l+1\}$, where $k,l\in \natu $ and $k<l$,
\end{itemize}
then $A$ is quasinormal.
It turns out that  we can replace the system of equations in condition (i) by single equation  \eqref{xrow} (cf. Theorem \ref{hyp}) and remove the assumption of separability in condition (ii) (cf. Theorem \ref{ucc}). Moreover, he obtains analogical results for densely defined operators on the assumption that some inclusion of domains of the operators $T^*T$, $TT^*$ and $T^*TP$, ($P$ is the projection onto the closure of range of $T$) holds. Several years later in the paper  \cite{uch2} on occasion of investigating  properties  class of log-hyponormal operators, he showed that single equality   \eqref{xrow}  implies  quasinormality in this class. This proof contains the proof of the fact that bounded operator in Hilbert space which satisfies  \eqref{xukl} with $S=\{2,3\}$ is quasinormal. As was mentioned in the paper  \cite{jabl} in the case of bounded operators, this
characterization has been known for specialists working in this area since the late '80s. This characterization was independently discovered by A. A. S. Jibril \cite[Proposition 13]{Jib}.  Unfortunately, this paper
contains several errors. Z. J. Jab{\l}o{\'n}ski, I. B. Jung, and J. Stochel \cite{jabl} extended this characterization to densely defined operators. Their proof makes use of  the technique of bounded vectors. They also gave three Examples of the non-quasinormal operator which satisfies the single equation \eqref{xrow}.
The first example is related to Toeplitz operators on the
Hardy space $H^2$, while two others are linked to weighted shifts on a directed tree. The class of  weighted shifts on a directed tree was introduced in \cite{memo} and intensively studied
since then \cite{9,geh,chav,planeta,bdp,adv,mart}. The class is a
source of interesting examples (see e.g., \cite{dym,9,ja,trep,triv,abs,sque,jabl}). 

In particular, they  showed that for every integer
$n\Ge 2$, there exists a weighted shift $A$ on a
rooted and leafless directed tree with one branching
vertex such that
   \begin{equation}\label{qqq}
\text{$(A^{*} A)^n=A^{*n}A^{n}$ and $(A^{*} A)^k \neq
A^{*k}A^{k}$ for all $k\in\{2,3,\ldots\}\setminus
\{n\}$.}
   \end{equation}
 \cite[Example 5.5]{jabl}.  
It remained an open question as to whether such
construction is possible on a rootless and leafless
directed tree. This is strongly related to the
question of the existence of a composition operator
$A$ in a $L^2$-space (over a $\sigma$-finite measure
space) which satisfies \eqref{qqq}. It is shown  that the answer is in the affirmative. In fact, the author showed that for every integer $n\Ge 2$ there exists a (necessarily non-quasinormal) weighted shifts
$A$ on a rootless and leafless directed tree with one
branching vertex which satisfies \eqref{qqq} (cf.\
 \cite[Theorem 5.3.]{ja}). This combined with the fact
that every weighted shift on a rootless directed tree
with nonzero weights is unitarily equivalent to a
composition operator in an $L^2$-space (see
\cite[Theorem 3.2.1]{memo} and \cite[Lemma 4.3.1]{9})
yields examples of composition operators satisfying. It was observed in \cite[Theorem 3.3]{ja} that in the class of bounded
injective bilateral weighted shifts, the single
equality \eqref{xrow} with $n\Ge 2$ does
imply quasinormality. This
is no longer true for unbounded ones even for $k=3$
(cf.\  \cite[Example 4.4.]{ja}).

In this paper we will show that  operator $A$ is quasinormal iff it satisfies  \eqref{xukl} with $S=\{p,m,m+p,n,n+p\}$ for $p,m,n\in\natu$ (cf. Theorem \ref{gll}). This  theorem generalizes a characterization of quasinormality of bounded
 operators given in \cite[Theorem 2.1.]{uch} and \cite[Proposition 13]{Jib}. The proof of this characterization  makes use of the theory of 
operator convex  functions and related to this theory Davis-Choi-Jensen inequality (cf. Theorem \ref{lohe}). In the case $S=\{p,q,p+q,2p,2p+q\}$ we will give an alternative  proof which is completely different and fits nicely into
our framework (cf. Theorem \ref{jjj}). Moreover, if $A$ is invertible then \eqref{xukl} with $S=\{m, n, n+m\}$ for $m,n\in \natu$ also implies quasinormality of $A$ (cf. Theorem \ref{inv}). We obtain a new characterization of the normal operators which resembles that for the quasinormal operators.

   \section{Preleminaries}
In this paper, we use the following notation. The
fields of rational, real and complex
numbers are denoted by $\mathbb{Q}$,
$\mathbb{R}$ and $\mathbb{C}$, respectively. The
symbols $\mathbb{Z}$, $\mathbb{Z}_{+}$, $\mathbb{N}$
and $\mathbb{R}_+$ stand for the sets of integers,
nonnegative integers, positive integers and
nonnegative real numbers, respectively.

All Hilbert spaces considered in this paper are assumed to be complex. Let $A$ be a linear operator in a complex Hilbert
space $\mathcal{H}$. Denote by  $A^*$  the adjoint
of $A$.  We write $\boldsymbol{B}(\mathcal{H})$ and $\boldsymbol{B}_+(\mathcal{H})$, for the
set of all bounded operators   and positive operators in $\mathcal{H}$, respectively.
The following fact  follows from The Spectral Theorem
and plays an important role in our further
investigations.
\begin{theorem} \label{rsn}$($cf. \cite{rud}$)$ Let $n\in \natu$.
Commutants of a positive operator and it's $n$-th root coincides.
\end{theorem}

A linear map $\varPhi : \mathcal{A} \rightarrow \mathcal{ B}$ between $C^
*$-algebras is said to be positive if $\varPhi(A) \geq 0$ whenever
$A \geq 0$. It is unital if  $\varPhi$ preserves the identity.

Let $J\subset \real$ be an interval. A function $f : J \rightarrow \real$ is said to be
\begin{itemize}
\item[(i)] \textit{matrix monotone of degree $n$ } or \textit{ $n$-monotone}, if, for every selfadjoint $n$-dimension matrix $A$ and $B$, where $n\in \natu$ with $\sigma(A),\sigma(B)\subset J$ inequality  $A\leq B$ implies $f(A)\leq f(B)$,
\item[(ii)] \textit{operator monotone} or \textit{matrix monotone}, if it is $n$-monotone for every $n\in \natu$,
\item[(iii)] \textit{matrix convex of degree $n$ } or \textit{ $n$-convex}, if for every selfadjoint $n$-dimension matrix $A$ and $B$, where $n\in \natu$ with $\sigma(A),\sigma(B)\subset J$ 
\begin{equation*}
    f(tA+(1-t)B)\leq tf(A)+(1-t)f(B) \quad \text{for all}\:\: t\in[0,1],
\end{equation*}
\item[(iv)] \textit{operator convex} or \textit{matrix convex}, if it is $n$-monotone for every $n\in \natu$.
\end{itemize}

In 1934 K. L\"owner \cite{l10} proved that a function defined on
an open interval is operator monotone, if and only if it allows an analytic
continuation into the complex upper half-plane, that is an analytic continuation to a Pick function. The class of operator monotone functions is an important
class of real-valued functions and it has various applications in other branches of mathematics. This concept is closely related to operator convex functions which was  introduced by F. Kraus in a  paper \cite{kraus}. The  operator monotone functions and operator convex functions have very important
 properties, namely, they admit integral representations with respect to suitable
Borel measures. In particular, a continuous function $f : [0, \infty) \rightarrow [0, \infty)$ is operator monotone if
and only if there is a finite Borel measure $\mu$ on $[0,\infty)$ such that
 $\int_0^\infty \frac{1}{1+\lambda^2}d\mu(\lambda)<\infty$ and
\begin{equation}\label{repbol}
f(t)=\alpha +\beta t+\int_0^\infty \Big(\frac{\lambda}{1+\lambda^2}-\frac{1}{t+\lambda}\Big)d\mu(\lambda),
\end{equation}
where $\alpha\in \mathbb{R}$ i $\beta \geq 0$. By
the Bendat-Sherman formula (see \cite{ben,hans2}) operator convex function $f:(-1,1)\rightarrow \real$ admits an integral representation
\begin{equation}\label{sher}
 f(t)=\alpha +\beta t+\int_{-1}^{1} \frac{t^2}{1-t\lambda}d\mu(\lambda),
\end{equation}
with $\alpha\geq0$  and $\mu$ is a positive measure. We give below example of a function which is operator monotone.
\begin{ex}\label{ex}
The  function $x^p$ for $p\in (0,1)$ is operator monotone and has an integral
representation

 \begin{equation*}
x^p=\frac{\sin p\pi}{\pi}\int_0^\infty \frac{x\lambda^{p-1}}{x+\lambda}d\mu(\lambda).
\end{equation*}

\end{ex}
The fact that function from Example \ref{ex} is operator monotone is well known as L\"owner-Heinz inequality.

\begin{theorem}$($L\"owner-Heinz inequality \cite{l9,l10}$)$. \label{lohe}
Let $A$, $B$ be  bounded positive  operators on    $\mathcal{H}$ such that $0\leq B \leq A$. If $0\leq p \leq 1$ then
$ B^p \leq A^p$.
\end{theorem}
The other two inequalities related to operator monotone and convex functions are Hansen inequality and Davis-Choi-Jensen inequality. The first of this has been established in \cite{hansen} by F. Hansen. In \cite{uch} M. Uchiyama gave a necessary and sufficient condition for the equality in the Hansen inequality and use it to show that \eqref{xukl} with $S=\{k,k+1,l,l+1\}$  implies quasinormality of $A$ in separable Hilbert space. The key ingredient of its
proof is the integral representation of operator monotone functions \eqref{repbol}.

\begin{theorem}$($Hansen inequality, cf. \cite{hansen, uch}$)$
\label{hans} Let $f:[0,\infty)\rightarrow \mathbb{R}$ be an operator monotone  function with $f(0)\geq 0$. Suppouse that $A$ is a bounded positive operator and $P$ a non trivial projection. Then we have  
\begin{equation*}
Pf(A)P\geq f(PAP).
\end{equation*}
Moreover the equality hold, only in the case of $PA=AP$ and $f(0)=0$, if $f$ is not a linear function.
\end{theorem}
Now we formulate the Jensen operator inequality (Davis-Choi-Jensen inequality)  due to Davis
 \cite{davi} and Choi \cite{choi}. In \cite{petz} D. Petz gave
 a necessary and sufficient condition
for the equality in the Jensen's operator inequality using integral representation of operator convex functions \eqref{sher}. 
\begin{theorem}$($Davis-Choi-Jensen inequality$)$\label{petz} Let $\mathcal{A}$ and $\mathcal{B}$ be $C^*$-algebras with unit and let
$\varPhi :\mathcal{A}\rightarrow \mathcal{B}$ be a unital positive linear map. If $f:(\alpha,\beta)\rightarrow \real$ is an operator convex function then
\begin{equation*}
f(\varPhi(a)) \leq \varPhi(f(a))    
\end{equation*}
for every $a=a^*\in \mathcal{A}$ with $\sigma(a)\subset (\alpha,\beta)$. Moreover, if $f$ is nonaffine then the equality holds if and only if $\varPhi$ restricted to the subalgebra generated by $\{a\}$ is multiplicative.

\end{theorem}

 \section{A characterization of quasinormality}
   
 In this section, we obtain new characterizations of quasinormality in terms of the system of equations \eqref{xukl}. We first show that  the assumption of separability of \cite[Theorem 2.1.]{uch} can be removed. The following lemma was suggested to us by Professor Jan Stochel.
 
  \begin{lemma}\label{yyy}
Let $S \subset \mathbb{N}$ be a nonempty set such that 
    \begin{align} \label{quoo}
\begin{minipage}{70ex} 
any bounded operator $A$ on a separable  Hilbert space that satisfies \eqref{xukl} is quasinormal. 
\end{minipage}
    \end{align}
Then \eqref{quoo} remains true for  any Hilbert space.
\end{lemma}
\begin{proof}
Let $\mathcal{H}$ be a  Hilbert space of any dimension. For every $f\in \mathcal{H}$, we define the separable subspace $\Mm_f$ of $\hil$ by
 \begin{equation*}
    \Mm_f:=\overline{\{{A^*}^{i_k}{A}^{j_k}\cdots{A^*}^{i_1}{A}^{j_1}f \colon (i_1, \ldots,i_k), (j_1, \ldots, j_k) \in \mathbb{Z}_+^k, \, k \in \natu\}}
 \end{equation*}
Note that  $\Mm_f$ reduces $A$ and operator $A|_{\Mm_f}$ also satisfies   \eqref{xukl}. Hence, by  \eqref{quoo}, $A|_{\Mm_f}$ is quasinormal. As a consequence, we have
 \begin{equation*}
   {A^*}|_{\Mm_f}{A}|_{\Mm_f}A|_{\Mm_f}= A|_{\Mm_f}{A^*}|_{\Mm_f}{A}|_{\Mm_f}.
 \end{equation*}
 In particular, applying the above
 to vector $f$, we see that
 \begin{equation*}
   {A^*}{A}Af= A{A^*}{A}f.
 \end{equation*}
 Since vector $f$ was chosen arbitrarily  $A$ is quasinormal.
\end{proof}

We are now ready to show that the assumption of separability of \cite[ Theorem 2.1.]{uch} can be removed. 

\begin{theorem}\label{ucc}
Let $k,l\in \natu$ be such that  $l<k$ and  $A$ be a bounded  operator on $\hil$. Then the following conditions are equivalent:
\begin{itemize}
 \item[(i)] operator $A$ is quasinormal,
\item[(ii)] operator $A$ satisfies  \eqref{xukl} with $S=\{l,l+1,k,k+1\}$.
\end{itemize}
\end{theorem}

\begin{proof}
Set $l,k\in \natu$ such that $l<k$. By  \cite[Theorem 2.1.]{uch}, the set $S:=\{l,l+1,k,k+1\}$ has the property \eqref{quoo} for separable  Hilbert space. Applying Lemma \ref{yyy} 
completes the proof.
\end{proof}

The proof of the main result of this section (cf. Theorem \ref{gll}) involves several lemmas. The first of which collects some facts related to the block decomposition  with respect to kernel/range decomposition of the operators which satisfies 
 \eqref{xukl}.

\begin{theorem}\label{ppppt}
Let $A$ be a bounded operator on  $\hh$. Consider the block decomposition
\begin{equation}\label{dek}
A=\left[ \begin{array}{cc} A^\prime &  0 \\R &  0 \end{array} \right] \quad\text{on} \quad \overline{\ran A^*)}\oplus \nul A)=\mathcal{H},
\end{equation} where $A^\prime:=P A|_{\overline{\ran A^*)}}$, $R:=QA|_{\nul A)}$, $P:=P_{\overline{\ran A^*)}}$ and $Q:=I-P$.   Then the following assertions are valid:
\begin{itemize}
\item[(i)] if  operator $A$ satisfies the equation \eqref{xrow} for some $n\geq2$, then $A^\prime$ is injective,
\item[(ii)] if operator $A$ satisfies  \eqref{xukl} with $S=\{k,k+1\}$ and $A^\prime$ is bounded below, then the following equality hold
\begin{equation}\label{onsr}
\|V^*(\M^*\M+R^*R)^{k}V\|=\|(\M^*\M+R^*R)^{k}\|,
\end{equation}
where $A^\prime=V|\M|$ is the polar decomposition of $\M$. 
\end{itemize}

\end{theorem}

\begin{proof}
(i) First, we choose $h\in \overline{\ran A^*)}$  such that $A^\prime h=0$, then
\begin{equation*}
\langle A^\prime h,f \rangle=0, \quad  f \in \overline{\ran A^*)}.
\end{equation*} 
Since $A^\prime=P A|_{\overline{\ran A^*)}}$, then the last line takes the form:
\begin{equation*}
\langle P Ah,f \rangle=0,  \quad f \in \overline{\ran A^*)},
\end{equation*} hence
\begin{equation*}
\langle Ah,A^*g\rangle=0, \quad  g \in \hh.
\end{equation*} We see
that the last condition is equivalent to $A^2h=0$. This and \eqref{xrow} (recall that $n\geq2$) give
\begin{equation*}
   (A^{*}A)^{n}h=A^{*n}A^{n}h=0,
\end{equation*}
which yields $h\in \nul(A^{*}A)^{n})=\nul |A|)=\nul A)$.
In turn, $h\in \overline{\ran A^*)}$, and so $h=0$. Hence Operator $ A^\prime$ is injective.

(ii) By \eqref{dek} we have
\begin{equation}\label{mod5}
A^*A=\left[ \begin{array}{cc} {A^\prime}^*A^\prime+R^*R &  0 \\0 &  0 \end{array} \right].
\end{equation}
 It is also easily seen that  \eqref{xukl} with $S=\{k,k+1\}$ implies  
\begin{equation*}
   (A^{*}A)^{k+1}=A^{*}(A^{*}A)^{k}A,
\end{equation*}
which  is equivalent to

\begin{align}\label{ro}
({A^\prime}^*A^\prime+R^*R)^{k+1}={A^\prime}^*({A^\prime}^*A^\prime+R^*R)^{k}A^\prime.
\end{align}
Note that
\begin{align}\label{to2}\|V^*(\M^*\M+R^*R)^{k}V\|&=\sup_{h\in \overline{\ran \M^*)},\|h\|=1} \langle V^*(\M^*\M+R^*R)^{k}Vh,h \rangle \\\notag&=\sup_{h\in \overline{\ran \M^*)},\|h\|=1} \langle (\M^*\M+R^*R)^{k}Vh,Vh \rangle
    \\\notag&=\sup_{g\in \overline{\ran V)},\|g\|=1} \langle (\M^*\M+R^*R)^{k}g,g \rangle
    \\\notag&=\sup_{f\in \overline{\ran A^*)}} \frac{\langle (\M^*\M+R^*R)^{k}\M f,\M f \rangle}{\|\M f\|^2}
    \\\notag&=\sup_{f\in \overline{\ran A^*)}} \frac{\langle \M^*(\M^*\M+R^*R)^{k}\M f,f \rangle}{\|\M f\|^2}
 \\\notag&=\sup_{f\in \overline{\ran A^*)}} \frac{\langle (\M^*\M+R^*R)^{k+1}f,f \rangle}{\|\M f\|^2}.
\end{align}
It follows from the selfadjointness of $(\M^*\M+R^*R)^{k}$ that  there exists $\lambda \in \mathbb{C}$ and sequence  $\{f_n\}_{n\in \natu}$ on $\hh$ such that the following conditions are satisfied:
    \begin{gather}
\label{dz1}
\|f_n\|=1, \quad n\in \natu,
    \\ \label{dz2}
|\lambda|=\|(\M^*\M+R^*R)^{k}\|,
    \\ \label{dz3}
\lim_{n\rightarrow \infty}\|(\M^*\M+R^*R)^{k}f_n-\lambda f_n\|=0.    
    \end{gather}
Since $\M$ is bounded below, there exists  $c>0$ such that  $\|\M f\|>c\|f\|$ for $f\in \overline{\ran A^*)}$.
Hence, we have
\begin{align*}
&\Big|\frac{\langle (\M^*\M+R^*R)^{k+1}f_n,f_n \rangle}{\|\M f_n\|^2}-\frac{\langle (\M^*\M+R^*R)\lambda f_n,f_n \rangle}{\|\M f_n\|^2}\Big|
\\&\leq\frac{1}{\|\M f_n\|^2}|\langle (\M^*\M+R^*R)((\M^*\M+R^*R)^{k}-\lambda) f_n,f_n \rangle|    
\\&\hspace{-1ex}\overset{\eqref{dz1}}\leq\frac{1}{c^2}\|\langle (\M^*\M+R^*R)((\M^*\M+R^*R)^{k}-\lambda) f_n\|
\\&\leq\frac{1}{c^2}\|\langle (\M^*\M+R^*R)\|\|((\M^*\M+R^*R)^{k}-\lambda)f_n\|.
\end{align*}
Letting $n$ to infinity in the above inequality and using \eqref{dz3}, we get 
\begin{equation} \label{to}
\frac{\langle (\M^*\M+R^*R)^{k+1}f_n,f_n \rangle}{\|\M f_n\|^2}-\frac{\langle (\M^*\M+R^*R)\lambda f_n,f_n \rangle}{\|\M f_n\|^2}\rightarrow 0.
\end{equation}
By inequality
\begin{equation*}
\frac{\langle (\M^*\M+R^*R) f_n,f_n \rangle}{\|\M f_n\|^2}\geq\frac{\langle (\M^*\M) f_n,f_n \rangle}{\langle (\M^*\M) f_n,f_n \rangle}=1
\end{equation*}
and \eqref{to}, we have 
\begin{equation*}
\sup_{f\in \overline{\ran A^*)}}\frac{\langle (\M^*\M+R^*R)^{k+1}f,f \rangle}{\|\M f\|^2}\geq |\lambda|\overset{\eqref{dz2}}=\|(\M^*\M+R^*R)^{k}\|.
\end{equation*} The last inequality with  \eqref{to2} gives
\begin{equation*}
\|V^*(M^*M+R^*R)^{k}V\|\geq\|(M^*M+R^*R)^{k}\|.
\end{equation*}
In turn
\begin{align*}
\|V^*(\M^*\M+R^*R)^{k}V\|&\leq\|V^*\|\|(\M^*\M+R^*R)^{k}\|\|V\|
\leq \|(\M^*\M+R^*R)^{k}\|.
\end{align*} Hence the equality in \eqref{onsr} holds.
\end{proof}

The following  Lemma  turns out to be useful. It will be used several times in this paper.

\begin{lemma}\label{pomoc}
 Let $k \in \mathbb{N}$ and $A$ be a bounded operator in $\hh$ such that  $A^*A$ commutes with   $A^k$ and satisfy the equation \eqref{xrow} with $n=k$.
Then $A$ is quasinormal.
\end{lemma}
\begin{proof}
The case of $k=1$ is obvious. Assume now that $k>1$. First we show that $A$ satisfy  \eqref{xukl} for $S=\{k,k+1,2k,2k+1\}$. Indeed, this is because
\begin{equation*}
    A^{*k+1}A^{k+1}=A^{*k}A^*AA^{k}=A^{*k}A^{k}A^*A=(A^*A)^{k+1}.
\end{equation*}
Similarly, we see that 
\begin{equation*}
    A^{*2k}A^{2k}=A^{*k}(A^*A)^kA^{k}=A^{*k}A^{k}(A^*A)^k=(A^*A)^{2k},
\end{equation*}
and
\begin{equation*}
    A^{*2k+1}A^{2k+1}=A^{*2k}A^*AA^{2k}=A^{*2k}A^{2k}A^*A=(A^*A)^{2k+1}.
\end{equation*}
Summarizing, we have shown that the operator $A$ satisfy  \eqref{xukl} with $S=\{k,k+1,2k,2k+1\}$. This and Theorem \ref{ucc}  imply that $A$
is quasinormal. 
\end{proof}

 For the reader's convenience, we include the proof of the following result which
is surely folklore.

\begin{lemma}\label{fug}
Let $M,N,T \in \textbf{B}(\mathcal{H})$ be such that $M$ and $N$ are positive  and
\begin{equation}\label{eq1}
TM^k=N^kT.
\end{equation}
Then $TM=NT$.
\end{lemma}
\begin{proof}
An application of Berberian's trick concerning $2\times2$ operator matrices gives that the following equation is equivalent to 
 \eqref{eq1}
\begin{equation*}
\left[ \begin{array}{cc} 0 &  0 \\T &  0 \end{array} \right] \left[ \begin{array}{cc} M & 0  \\0 &   N \end{array} \right]^k=\left[ \begin{array}{cc} M & 0  \\0 &   N \end{array} \right]^k \left[ \begin{array}{cc} 0 & 0 \\T &  0 \end{array} \right].
\end{equation*}
Using  Theorem \ref{rsn}, we get
\begin{equation*}
\left[ \begin{array}{cc} 0 &  0 \\T &  0 \end{array} \right] \left[ \begin{array}{cc} M & 0  \\0 &   N \end{array} \right]=\left[ \begin{array}{cc} M & 0  \\0 &   N \end{array} \right] \left[ \begin{array}{cc} 0 & 0 \\T &  0 \end{array} \right].
\end{equation*}
Looking more closely at the last line, we get  $TM=NT$ which completes the proof.
\end{proof}

We need one more fact in the proof of the main result of this section (cf. Theorem \ref{gll}), which seems to be of
independent interest.
\begin{lemma}\label{wak}
Let $\alpha, \beta\in\real_+$ be such that $\alpha<\beta$ and $A$, $B$ be bounded  operators on $\hil$ such that $B$ is positive and injective. Then the following conditions are equivalent:
\begin{itemize}
\item[(i)] $A^*A\leq B$ and $A^*B^sA=B^{s+1}$ for $s=\alpha,\beta$,
\item[(ii)] $A$ and $B$ commute and $A^*A= B$.
\end{itemize}
\end{lemma}
\begin{proof}

(ii)$\Rightarrow$ (i) This part is obvious.

(i)$\Rightarrow$ (ii) We conclude from the Douglas factorization lemma \cite[Theorem 1]{doug}
that there exist operator $Q$ such that $\|Q\|\leq 1$ and $A=QB^\frac{1}{2}$. This, injectivity of $B$ and condition (i) gives that
\begin{equation}\label{vvv}
Q^*B^sQ=B^{s}\quad \text{for}\quad s=\alpha,\beta.
\end{equation}
Consider the operators on $\hil\oplus\hil$ given by
\begin{align}Z=\left[ \begin{array}{cc} B^{\beta} &  0 \\0 &  0 \end{array}\label{hanpet} \right],\quad U=\left[ \begin{array}{cc} Q &  R \\S &  -Q^* \end{array} \right],\quad V=\left[ \begin{array}{cc} Q &  -R \\S &  Q^* \end{array} \right],
\end{align}
where $R=(I-QQ^*)^\frac{1}{2}$ and $S=(I-Q^*Q)^\frac{1}{2}$, and the maps
\begin{equation*}
    \varPhi:B(\hil\oplus\hil)\rightarrow B(\hil\oplus\hil) \quad\text{and}\quad \varPsi:B(\hil\oplus\hil)\rightarrow B(\hil)
\end{equation*} defined by
\begin{equation*}
    \varPhi(X)=\frac{1}{2}(U^*XU+V^*XV) \quad\text{and}\quad \varPsi(X)=P_{\hil\oplus\{0\}}\varPhi(X)|_{\hil\oplus\{0\}}.
\end{equation*}
  Using Lemma \ref{fug} with $n=2$ we verify that operators $U$ and $V$ are unitaries. Hence both $\varPhi$ and $\varPsi$ are unital positive linear maps. Let  $f:(0,\infty)\rightarrow\real$ be function given by $f(x)=x^\frac{\alpha}{\beta}$. We therefore have 
\begin{align*}
\varPsi(f(Z))&=P_{\hil\oplus\{0\}}\varPhi(f(Z))|_{\hil\oplus\{0\}}=P_{\hil\oplus\{0\}}\varPhi\big(\big[
\begin{smallmatrix}B^{\alpha} &  0 \\0 &  0 \end{smallmatrix}\big]\big)|_{\hil\oplus\{0\}}\\&=P_{\hil\oplus\{0\}}\Big[ \begin{smallmatrix} Q^*B^{\alpha}Q &  0 \\0 &  RB^{\alpha}R \end{smallmatrix}\Big]|_{\hil\oplus\{0\}}
=Q^*B^{\alpha}Q\overset{\eqref{vvv}}=(Q^*B^{\beta}Q)^\frac{\alpha}{\beta}\\&=f(Q^*B^{\beta}Q)=f\Big(P_{\hil\oplus\{0\}}\left[ \begin{smallmatrix}Q^*B^{\beta}Q &  0 \\0 &  RB^{\beta}R \end{smallmatrix}\right]|_{\hil\oplus\{0\}}\Big)\\&=f\big(P_{\hil\oplus\{0\}}\varPhi\big(\big[ \begin{smallmatrix} B^{\beta} &  0 \\0 &  0 \end{smallmatrix}\big]\big)|_{\hil\oplus\{0\}}\big)=f(\varPsi(Z)).
\end{align*}
Since $-f$ is operator convex and $\varPsi (-f(Z))=-f(\varPsi(Z))$ application of Theorem \ref{petz} shows that map $\varPsi$ restricted to the subalgebra generated by $\{Z\}$ is multipticative. In particular 
\begin{equation*}
    \varPsi(Z^k)=(\varPsi(Z))^k \quad \text{for every}\quad k\in \natu.
\end{equation*}
The last equality is equivlent to the following one
\begin{equation*}
    Q^*B^{k\beta}Q=(Q^*B^{\beta}Q)^k \quad \text{for every}\quad k\in \natu.
\end{equation*}
This and \eqref{vvv} gives
\begin{equation}\label{vvv2}
    Q^*B^{k\beta}Q=B^{k\beta} \quad \text{for every}\quad k\in \natu.
\end{equation}
Put $C:=B^\frac{\beta}{2}Q$. We deduce from \eqref{vvv2} that
\begin{align*}
    C^*C&=Q^*B^{\beta}Q=B^{\beta},
    \\C^{*2}C^2&=Q^*B^\frac{\beta}{2}C^*CB^\frac{\beta}{2}Q=Q^*B^{2\beta}Q=B^{2\beta}=(C^*C)^2,\\C^{*3}C^3&=C^{*}(C^{*2}C^2)C=Q^*B^\frac{\beta}{2}B^{2\beta}B^\frac{\beta}{2}Q\\&=Q^*B^{3\beta}Q=B^{3\beta}=(C^*C)^3.
\end{align*}
Since $C^{*2}C^2=(C^*C)^2$ and $C^{*3}C^3=(C^*C)^3$, we get that $C$ is quasinormal (cf. Theorem \ref{ucc}). Hence
\begin{equation*}
(C^*C)C=C(C^*C),
\end{equation*} which implies
\begin{equation*}
B^\beta B^\frac{\beta}{2}Q=B^\frac{\beta}{2}QB^\beta.
\end{equation*}
Since $B^\frac{\beta}{2}$ is injective we get
\begin{equation*}
    B^\beta Q=QB^\beta.
\end{equation*}
 We infer from Theorem \ref{rsn}
 that
 \begin{equation*}
    B Q=QB.
\end{equation*}
Multiplying the above equation right by $B^\frac{1}{2}$ and using $A=QB^\frac{1}{2}$, we see that $A$ and $B$ commute. Hence, we have
\begin{equation*}
    B^\beta A^*A= A^*B^\beta A=B^{\beta+1}.
\end{equation*}
This and injectivity of $B^{\beta}$ lead to 
\begin{equation*}
     A^*A=B,
\end{equation*}
which completes the proof.
\end{proof}
\begin{rem}
The operators defined in \eqref{hanpet}
originally appeared in \cite{hans3} on the ocassion of proving that if $f$ is operator convex function, then the inequality $f(A^*XA)\leq A^*f(X)A$ holds for any contractive operator $A$ and for any selfadjoint operator $X$.
\end{rem}

We are now in a position to formulate and prove the aforementioned analog
of Theorem \ref{ucc}.

\begin{theorem}\label{gll}
Let $m,n,p\in \mathbb{N}$ be such that $m< n$ and $A$ be a bounded operator in $\hil$. Then the following conditions are equivalent:
\begin{itemize}
    \item[(i)] operator $A$ satisfies \eqref{xukl} with $S=\{p,m,m+p,n,n+p\}$,
    \item[(ii)] operator $A$ is quasinormal.
\end{itemize}  
\end{theorem}
\begin{proof}
Consider the block decomposition as in Theorem \ref{ppppt}. Hence, by Theorem \ref{ppppt} (i), $\M$ is injective. An induction argument shows that 
\begin{align}\label{pot5}
A^k=\left[ \begin{array}{cc} A^{\prime k} &  0 \\RA^{\prime k-1} &  0 \end{array} \right], \quad k\in \natu.
\end{align}
It follows from  \eqref{xukl} with $S=\{m,m+p,n,n+p\}$ that
\begin{equation*}
    A^{*p}(A^*A)^sA^p=(A^*A)^{p+s}\quad \text{for}\quad s=m,n.
\end{equation*}
This, together with \eqref{mod5} and \eqref{pot5}, implies that 
\begin{equation}\label{h1}
    \M^{*p}(\M^*\M+R^*R)^s\M^p=(\M^*\M+R^*R)^{p+s}\quad \text{for}\quad s=m,n.
\end{equation}
We infer from \eqref{xrow} with $n=p$, \eqref{pot5} and \eqref{mod5} that
\begin{equation}\label{h2}
    \M^{*p}\M^p\leq \M^{*p-1}(\M^{*}\M+R^*R)\M^{p-1}=(\M^*\M+R^*R)^{p}.
\end{equation}
By \eqref{h1} and \eqref{h2} and Lemma \ref{wak} applied to the pair  $({\M}^p,({\M}^*{\M}+R^*R)^{p})$ in place of $(A,B)$ and with $\alpha=\frac{m}{p}$ and $\beta=\frac{n}{p}$, we deduce that 
\begin{align}\label{h3}
    \M^{*p}\M^p&=(\M^*\M+R^*R)^{p}
    \\ \label{h4}
 \M^p (\M^*\M+R^*R)^{p}&=  (\M^*\M+R^*R)^{p}\M^p.
\end{align}
 Combining \eqref{h2} with \eqref{h3}, we get
\begin{equation*}
    \M^{*p-1}R^*R\M^{p-1}=0,
\end{equation*}
which yields $R\M^{p-1}=0$.
This together with  \eqref{pot5} implies that
\begin{equation}\label{dekze}
A^p=\left[ \begin{array}{cc} A^{\prime p} &  0 \\0 &  0 \end{array} \right]. 
\end{equation}
We deduce from  \eqref{h4}, \eqref{dekze},  \eqref{mod5} and Lemma \ref{rsn} that operator $A^p$ commute with $A^*A$. Applying Lemma \ref{pomoc}
completes the proof.
\end{proof}
 The author proved that in the class of bounded injective bilateral weighted shifts, the single equation  \eqref{xrow}
with $n \geq 2$ does imply quasinormality (cf. \cite[Theorem
3.3]{ja}). As shown in   \cite[Example 3.4]{ja}, this no longer true for unbounded injective bilateral weighted shifts. Now we  
will show that the condition \eqref{xukl} with $S=\{m,n\}$, where  $n>m\geq2$, does imply quasinormality in the class of unbounded injective bilateral weighted shifts.
\begin{theorem}
Let $m,n\geq2$ be such that $m< n$. Then any injective  bilateral weighted shift
$A$ that satisfies \eqref{xukl} with $S=\{m,n\}$
is quasinormal.
\end{theorem}
\begin{proof}
Let $A$ be a  injective bilateral weighted shift with weights $\{\lambda_l\}_{l=-\infty}^\infty$ and $\{e_l\}_{l=-\infty}^\infty$ be the standard $0$-$1$  orthonormal basis of $\ell^2(\mathbb{Z})$. 
 By \cite[Theorem 3.2.1]{memo}, we can assume without loss of generality that $\lambda_l > 0$ for
all $l \in \mathbb{Z}$. Suppose that $A$ satisfies \eqref{xukl} with $S=\{m,n\}$. It is easy to verify that the equation
\begin{equation*}
    A^{*k}A^ke_l=(A^{*}A)^ke_l,\quad l\in \mathbb{Z}, \, k\in\{m,n\},
\end{equation*}
implies that 
\begin{equation}
    \lambda_l^k=\lambda_{l}\lambda_{l+1}\cdots\lambda_{l+k-1},\quad l\in \mathbb{Z},\, k\in\{m,n\}.
\end{equation}
Hence, the sequence $\{a_l\}_{l=-\infty}^\infty$, where $a_l:=\ln \lambda_l$  satisfies the following recurrence relation for every $k\in \{m,n\}$,
\begin{equation} \label{dus1}
   (k-1)a_l=a_{l+1}+\cdots+a_{l+k-1},\quad l\in \mathbb{Z}.
\end{equation} 
Observe that for every $k\in \{m-1,n-1\}$, the polynomial $\frac{1}{k}p_k(z)$ defined by 
\begin{equation*}
p_k(z):= kz^k - (z^
{k-1} + z^{k-2}+\dots+ 1),\quad z\in \mathbb{C},   \end{equation*} 
is  the characteristic polynomial of the recurrence relation \eqref{dus1}. Since 
\begin{equation}
    \frac{p_k(z)}{z-1}=\frac{d}{dz}\frac{z^{k+1}-1}{z-1}
\end{equation}
we infer from \cite[Example 3.7]{kaj} that  the polynomials $p_k$ for $k\in \{m-1,n-1\}$, have only one common
root $z=1$. This combined with \cite[Lemma 4.1]{ja} and \cite[ Theorem 3.1.1]{hal} implies that the sequence $\{a_l\}_{l=-\infty}^\infty$ is constant. Hence, the operator $A$ is a multiple of a unitary operator and as such is quasinormal. This completes
the proof.
\end{proof}

In Theorem \ref{nowechar} below, we propose
a method of constructing new sets $S$ such that \eqref{xukl} characterize quasinormal operators.

\begin{theorem}\label{nowechar}
Let $n\in \natu$ and   $S\subset \natu$ be a nonempty set   such that
any bounded operator $A$ on $\hil$ that satisfies \eqref{xukl} is quasinormal.
 Then set $\hat{S}=\{n\} \cup \{ns\:| s\in S\}$ also  has this property.
\end{theorem}
\begin{proof}
Suppose $A$ is operator which satisfies 
\eqref{xukl} with $S=\hat{S}$.
It follows that
\begin{equation*}
   (A^n)^{*s}(A^n)^{s}=A^{*sn}A^{sn}=(A^{*}A)^{sn}= ((A^{*}A)^{n})^s=(A^{*n}A^n)^{s}.
\end{equation*}
Since the operator $A^n$ satisfies  \eqref{xukl} with subset  $S$, hence it is quasinormal. Thus we obtain
\begin{equation}\label{nam}
    A^{*n}A^nA^n=A^nA^{*n}A^n.
\end{equation}
Operator $A$ also satisfy  equation \eqref{xrow}. We conclude from  \eqref{nam} and  Theorem  \ref{rsn}  that operator $A^n$ commutes  with $A^*A$.
Now applying Lemma \ref{pomoc} completes the proof.
\end{proof}

Now we give another proof of Theorem \ref{gll} in the case of $S=\{p,q, p+q, 2p, 2p+q\}$. It is worth noting that this proof uses only elementary properties of $C^*$algebras and Theorem \ref{rsn}. We do not use theory of operator monotone and convex functions.   We begin by proving the following lemma:
\begin{lemma}\label{komut}
Let $p,q\in \mathbb{N}$ and $A$ be bounded operator on $\hh$ which satisfies  \eqref{xukl} with  $S=\{p,q, p+q, 2p, 2p+q\}$. Then operator $A^q$ commute with $A^{*p}A^{p}$.
\end{lemma}

\begin{proof}

It follows from our
assumptions that the following chain of
equalities holds:
\begin{align}\label{rec}
A^{*(p+q)} A^{p}A^{*p}A^{p+q}&=A^{*q}(A^{*p}A^{p})^2A^q=A^{*q}(A^{*}A)^{2p}A^q
\\&
=A^{*q}(A^{*2p}A^{2p})A^q=A^{*2p+q}A^{2p+q}=(A^*A)^{2p+q}. \notag
\end{align}
We show that operators  $A^q$ and $A^{*p}A^{p}$ commute. Indeed, if $f\in \mathcal{H}$, then
\begin{align*}
\|(A^{*p}A^{p+q}&-A^qA^{*p}A^{p})f\|^2=\langle A^{*p}A^{p+q}f,A^{*p}A^{p+q}f\rangle
\\&-2\re \langle A^{*p}A^{p+q}f,A^qA^{*p}A^{p}f\rangle
+\langle A^qA^{*p}A^{p}f,A^qA^{*p}A^{p}f\rangle
\\&
=\langle A^{*(p+q)} A^{p}A^{*p}A^{p+q}f,f\rangle
-2\re\langle A^{*p}A^{p}A^{*q}A^{*p}A^{p+q}f,f\rangle
\\&
+\langle A^{*p}A^{p}A^{*q}A^qA^{*p}A^{p}f,f\rangle
\\&\overset{\eqref{rec}}
=\langle(A^*A)^{2p+q} f,f\rangle
-2\re\langle A^{*p}A^{p}A^{*(p+q)}A^{p+q}f,f\rangle
\\&
+\langle A^{*p}A^{p}(A^{*q}A^q)A^{*p}A^{p}f,f\rangle
\\&
=\langle(A^*A)^{2p+q} f,f\rangle
-2\re\langle (A^{*}A)^{p}(A^{*}A)^{p+q}f,f\rangle
\\&
+\langle (A^{*}A)^{p}(A^{*}A)^q(A^{*}A)^{p}f,f\rangle=0,
\end{align*}
which implies
\begin{equation*}
A^{*p}A^{p+q}=A^qA^{*p}A^{p}.
\end{equation*}
This completes the proof.
\end{proof}
As an immediate consequence of Lemma \ref{komut}  we obtain the next Theorem.
\begin{theorem}\label{jjj}
Let $p,q\in \mathbb{N}$ and $A$ be bounded operator in $\hh$ which satisfies  \eqref{xukl} with  $S=\{p,q, p+q, 2p, 2p+q\}$. Then operator $A$ is quasinormal.
\end{theorem}
\begin{proof}
Applying Lemma \ref{komut}, we deduce that   
\begin{equation*}
A^qA^{*p}A^{p}=A^{*p}A^{p}A^q.
\end{equation*}
This, equation \eqref{xrow} with $n=p$ and  Theorem \ref{rsn} yields
\begin{equation}\label{www}
A^qA^{*}A=A^{*}A^{}A^q.
\end{equation}
To complete the proof, it suffices to consider two disjunctive cases.

\textsc{Case} 1.
If $q=1$, then by \eqref{www} operator  $A$ is quasinormal. 

\textsc{Case} 2. We now consider the other case when $q>1$, then it follows from \eqref{www} that
\begin{align*}
    A^{*2q}A^{2q}&=A^{*q}A^{q*}A^qA^q=A^{*q}(A^{*}A)^qA^q\\&=A^{*q}A^q(A^{*}A)^q=(A^{*}A)^{2q}\\
    A^{*sq+1}A^{sq+1}&=A^{*sq}A^{*}AA^{sq}=A^{*sq}A^{sq}A^{*}A=(A^{*}A)^{sq+1},
\end{align*}
for $s=1,2$. Hence  operator $A$ satisfies  \eqref{xukl}  with  $S=\{q,q+1,2q,2q+1\}$. This and   Case 1. imply that   operator $A$ is quasinormal.
\end{proof}

The new operator transform $\hat{A}$ 
of $A$ from the class $A(k)$ to the class of hyponormal operators was introduced  in
\cite{mary}. We define $\hat{A}$ by
\begin{equation*}
 \hat{A}:=WU||A|^kA|^\frac{1}{k+1},   
\end{equation*}
where $|A||A^*|=W||A||A^*||$ is the polar decomposition. The following Theorem guarantees  $\hat{A}$ is indeed hyponormal.

\begin{theorem}$($cf. \cite[ Theorem 1]{mary} $)$.\label{pop}
Let $A=U|A|$ be the polar decomposition of  class $A(k)$ operator, then operator $\hat{A}:=WU||A|^kA|^\frac{1}{k+1}$ is hyponormal, where $|A||A^*|=W||A||A^*||$ is the polar decomposition.
\end{theorem}

The following characterization of quasinormality  can be deduced from Theorem \ref{pop} and \cite[Proposition 2.3.]{uch}.

\begin{theorem}
Let $k\in \natu$. Bounded operator $A$ is quasinormal iff it satisfies \eqref{xukl} with $S=\{k,k+1\}$ and $ |A||A^*|=P_{\overline{\ran A)}}||A||A^*||$ is the polar decomposition.
\end{theorem}
\begin{proof}
Note that if $A$ is quasinormal, then it   \eqref{xukl} for $S=\{k,k+1\}$ (see \eqref{wst}). Let $A=U|A|$ be the polar decomposition. Since $|A^*|= U|A|U^*$ yields
 \begin{equation*}
   |A||A^*|= |A|U|A|U^*=U|A|^2U^*=(U|A|U^*)^2=|A^*|^2
\end{equation*}
we deduce that
\begin{equation*}
    |A||A^*|=P_{\overline{\ran A)}}|A^*|^2
\end{equation*}
is the polar decomposition.

We now show the reverse implication. Since $A$ satisfy  \eqref{xukl} with $S=\{k,k+1\}$, we see that $(A^*|A|^{2k}A)^\frac{1}{k+1}\geq |A|^2$.
Hence, 
operator $A$ is in class $A(k)$. By Theorem \ref{pop} the transform  $\hat{A}$ is  hyponormal operator. By our assumption, we have
\begin{equation*}
   \hat{A}=WU||A|^kA|^\frac{1}{k+1}=P_{\overline{\ran A)}}U|A|=A.
\end{equation*}
This yields $A$ is also hyponormal operator. This and the fact that   $A$ satisfy  \eqref{xukl}  with $S=\{k,k+1\}$ combined with \cite[Proposition 2.3.]{uch} completes the proof.
\end{proof}
\section{A characterization of normality}

The single equality \eqref{xrow}  implies the normality of compact operators \cite[Preposition]{uch}. In general, neither  the single equality \eqref{xrow} nor the system of equations \eqref{xukl} imply the normality. We obtain a new characterization of the normal operators which resembles that for the quasinormal operators (cf.  Theorem \ref{gll}).

We begin by proving a few auxiliary lemmas.
The first one will be deduced from  Lemma \ref{fug}.
\begin{lemma}\label{przep}
Let $k\in\mathbb{N}$, and $A$ be a  bounded operator such that both  operators $A$ and $A^*$ ssatisfies the equality  \eqref{xrow}. Then the following equation hold:
\begin{equation*}
AA^{*}A^k=A^kA^{*}A
\end{equation*}
\end{lemma}
\begin{proof}
It is clear that 
\begin{equation*}
A^k(A^{*k}A^{k})=(A^kA^{*k})A^{k}.
\end{equation*}
Using equalities $A^{*k}A^{k}=(A^{*}A)^{k}$ and $A^{k}A^{*k}=(AA^{*})^{k}$ we obtain the equality 
\begin{equation}\label{r1}
A^k(A^{*}A)^{k}=(AA^{*})^kA^{k}.
\end{equation}
Employing Lemma \ref{fug} completes the proof. 
\end{proof}
We will derive Theorem \ref{md} from the following more technical result.

\begin{lemma}\label{gl}
Let $k\in\mathbb{N}$ and  $A$ be a bounded operator in $\hh$ such that both operators $A$ and $A^{*}$  satisfies \eqref{xukl} with $S=\{k,k+1\}$. 
Then $A$ is quasinormal.
\end{lemma}
\begin{proof}
Since both operators $A$ and  $A^*$ satisfies  \eqref{xukl}, we can deduce from Lemma \ref{przep} the following equality
\begin{equation*}
AA^{*}A^{k+1}=A^{k+1}A^{*}A.
\end{equation*}
An induction argument shows that for every natural number j the
equality holds

\begin{equation*}
(AA^{*})^jA^{k+1}=A^{k+1}(A^{*}A)^j,
\end{equation*}
for every $j\in \natu$.
In particular for  $j=k$ we have
\begin{equation*}
(AA^{*})^kA^{k+1}=A^{k+1}(A^{*}A)^k.
\end{equation*} This and equation \eqref{xrow} with $n=k$ for operators $A$ and $A^*$ gives
\begin{equation}\label{recrel}
A^kA^{*k}A^{k+1}=A^{k+1}A^{*k}A^k.
\end{equation}
We now prove that operators $A$ and $A^{*k}A^{k}$ commute.
\begin{align*}
\|(A^{*k}A^{k+1}&-AA^{*k}A^{k})f\|^2=\langle A^{*k}A^{k+1}f,A^{*k}A^{k+1}f\rangle
-2\re \langle A^{*k}A^{k+1}f,AA^{*k}A^{k}f\rangle
\\&
+\langle AA^{*k}A^{k}f,AA^{*k}A^{k}f\rangle
\\&
=\langle A^{*(k+1)} (A^{k}A^{*k}A^{k+1})f,f\rangle
-2\re\langle A^{*k}A^{k}A^{*}A^{*k}A^{k+1}f,f\rangle
\\&
+\langle A^{*k}A^{k}A^{*}AA^{*k}A^{k}f,f\rangle
\\&\overset{\eqref{recrel}}
=\langle A^{*(k+1)} (A^{k+1}A^{*k}A^{k})f,f\rangle
-2\re\langle A^{*k}A^{k}A^{*(k+1)}A^{k+1}f,f\rangle
\\&
+\langle A^{*k}A^{k}(A^{*}A)A^{*k}A^{k}f,f\rangle
\\&
=\langle (A^{*}A)^{k+1}(A^{*}A)^{k}f,f\rangle
-2\re\langle (A^{*}A)^{k}(A^{*}A)^{k+1}f,f\rangle
\\&
+\langle (A^{*}A)^{k}(A^{*}A)(A^{*}A)^{k}f,f\rangle=0,
\end{align*}
for every $f\in \mathcal{H}$.
Hence,
\begin{equation*}
A^{*k}A^{k+1}=AA^{*k}A^{k}.
\end{equation*}
Using \eqref{xrow} with $n=k$, we get
\begin{equation*}
(A^{*}A)^{k}A=A(A^{*}A)^{k}.
\end{equation*}
By  Theorem \ref{rsn}, we see that
\begin{equation*}
A^{*}A^{2}=AA^{*}A
\end{equation*}
which yields that  operator $A$ is quasinormal. This completes the proof.
\end{proof}
Now we are in a position to prove the aforementioned characterization of normality
of  operators, which is a direct consequence of the above Lemma.
\begin{theorem}\label{md}
Let $k \in \mathbb{N}$, and $A$ be a bounded operator in $\hil$. Then the following conditions are equivalent:
\begin{enumerate}
\item[(i)] both operator $A$ and $A^{*}$ satisfies  \eqref{xukl} with $S=\{k,k+1\}$,
\item[(ii)] operator $A$ is  normal.
\end{enumerate}
\end{theorem}
\begin{proof}
Employing Lemma \ref{gl} to operator $A$ and $A^{*}$, we get, that both of them are quasinormal. Since  quasinormal operator is hyponormal, we  see that operator $A$ and $A^{*}$ are hyponormal hence,
\begin{equation*}
\|Af\|=\|A^{*}f\|.
\end{equation*}
 By \cite[Theorem 12.12]{rud} $A$  is normal. 
\end{proof}

We conclude this section by giving an analogue of Theorem \ref{gll} in case of an invertible operator.

\begin{theorem}\label{inv}
Let $m,n\in \natu$ be such that $m\leq n$ and $A$ be an invertible operator. Then the following conditions are equivalent:
\begin{itemize}
\item[(i)] operator $A$ and $A^{*}$ satisfies  \eqref{xukl} with $S=\{m,n,m+n\}$,
\item[(ii)] operator $A$ is normal.
\end{itemize}
\end{theorem}
\begin{proof}
 Let $A^n=U_n|A^n|$ be  the polar decomposition of $A^n$. Since $A^n$ is invertible operator, we conclude that  $U_n$ is unitary operator. By our assumption,
 \begin{equation}\label{inj}
     A^{*n}(A^*A)^mA^{n}=(A^*A)^{m+n}
 \end{equation} 
 It follows from \eqref{inj} and  the polar decomposition of $A^n$ that
 \begin{equation}
     |A^{n}|U_n^*(A^*A)^mU_n |A^{n}|=(A^*A)^{m+n}.
 \end{equation}
 The injectivity of $|A^{n}|$ implies
 \begin{equation}
     U_n^*(A^*A)^mU_n =(A^*A)^m,
 \end{equation}and consequently $(A^*A)^m$ and $U_n$ commute. By Theorem \ref{rsn},  $|A|^n$ and $U_n$ commute. Since $|A|^n=|A^n|$,  $A^n$ is quasinormal. Applying Theorem \ref{rsn} once more we see that $A^*A$ and $A^n$ commute. This and Lemma \ref{pomoc} show that $A$ is quasinormals. Since invertible quasinormal operators are normal the proof is complete. 
 \end{proof}

\section{Operator inequalities}

 Aluthge and Wang \cite{aw3,aw4} showed several results on powers
of $p$-hyponormal and log-hyponormal operators. The study has been
continued by  Furuta i Yanagida \cite{fy1,fy2}, Ito \cite{ito} and Yamazaki \cite{yama}. We collect this results 

\begin{theorem}\label{ph}
Let $m\in \natu$ and $A$ be $p$-hyponormal operator with $p\in (m-1,m]$. Then the following inequalities holds
\begin{itemize}
 \item[(i)] $A^{*n}A^{n}\geq(A^{*}A)^{n}$ and $(AA^{*})^{n}\leq A^{n}A^{*n}$, for every positive integer $n\leq m$,
 \item[(ii)]
 \begin{align*}
   (A^{*n}A^{n})^{\frac{(p+1)}{n}}\geq \dots &\geq  (A^{*m+2}A^{m+2})^{\frac{(p+1)}{m+2}}\\&\geq (A^{*m+1}A^{m+1})^{\frac{(p+1)}{m+1}}\geq (A^{*}A)^{p+1}
 \end{align*}
 and
 \begin{align*}
   (A^{n}A^{*n})^{\frac{(p+1)}{n}}\leq \dots  &\leq (A^{m+2}A^{*m+2})^{\frac{(p+1)}{m+2}}\\&\leq (A^{m+1}A^{*m+1})^{\frac{(p+1)}{m+1}}\leq (AA^{*})^{p+1},
 \end{align*}
 for $n\geq m+1$.
\end{itemize}
\end{theorem}and analogical results for  log-hyponormal operators.

\begin{theorem}
Let $A$ be log-hyponormalnym operator. Then
 \begin{equation*}
   (A^{*n}A^{n})^{\frac{1}{n}}\geq \dots  \geq (A^{*3}A^{3})^{\frac{1}{3}}\geq (A^{*2}A^{2})^{\frac{1}{2}}\geq (A^{*}A)
 \end{equation*}
 and
 \begin{equation*}
   (A^{n}A^{*n})^{\frac{1}{n}}\leq \dots  \leq (A^{3}A^{*3})^{\frac{1}{3}}\leq (A^{2}A^{*2})^{\frac{1}{2}}\leq (AA^{*})
 \end{equation*}
 dla $n\in\natu$.
 \end{theorem}
 The following Theorem which is a reinforcement
of \cite[Preposition 2.3.]{uch} is an immediate consequence of Theorem \ref{ph}.

 \begin{theorem}\label{hyp} Let $m,n\in \natu$ and $A$ be $p$-hyponormal operator with  $p\in (m-1,m]$.
 If $A$ satisfies equation \eqref{xrow} with $n\geq m+3$ then  is quasinormal.
 Morover, if operator $A$ is hyponormal and satisfies equation with $k\geq 2$ then is quasinormal. 
 \end{theorem}
 \begin{proof}
By \eqref{xrow} for $n=m+3$ and  the condition (ii) of Theorem \ref{ph},  operator $A$ satisfy  \eqref{xukl} with  $S=\{m+1,m+2,m+3\}$.
This combined with Theorem  \ref{ucc} implies that $A$ is quasinormal.

In the case when $A$ is  hyponormal, using once again  the condition (ii) of  Theorem \ref{ph}, we deduce that  $A$ satisfy the equation $A^{*2}A^2=(A^*A)^2$. The last equality is equivalent to the following one $A^*(A^*A-AA^*)A=0$. Since  $A$ is hyponormal yields $A^*A-AA^*$ is non-negative, we deduce that 
$(A^*A-AA^*)^\frac{1}{2}A=0$. This completes the proof.
 \end{proof}
 
The following result is also an analog of Theorems \ref{ph}.
 \begin{theorem}$($cf. \cite{mia}$)$.\label{mia}
 Let $A$ be invertible operator class $A$.  Then the following inequalities holds:
 \begin{itemize}
  \item[(i)] $|A^n|^\frac{2}{n}\geq (A^*|A^{n-1}|^\frac{2}{n-1}A)\geq |A|^2$, \quad  $n=2,3,\dots,$
  \item[(ii)] $|A^{n+1}|^\frac{2n}{n+1}\geq |A^n|^2$, \quad $n\in \natu$,
  \item[(iii)] $|A^{2n}|\geq |A^n|^2$, \quad $n\in \natu$,
  \item[(iv)] $|A|^2\geq |A^2|\geq \dots \geq |A^n|^\frac{2}{n}$, \quad $n\in \natu$,
  \item[(v)] $|A^{-2}|\geq |A^{-1}|^2$.

 \end{itemize}
 \end{theorem}
 
The key ingredient of its
proof consists the next Lemma, which
is a direct consequence of  the celebrated Furuta inequality (cf.  \cite{furuta}).
 
 \begin{theorem}\label{twc}$($cf. \cite{mia}$)$ Let $A$ and $B$ be positive invertible operators such that $(B^\frac{1}{2}AB^\frac{1}{2})^\frac{\beta_0}{\alpha_0+\beta_0}\geq B$ holds for fixed $\alpha_0 , \beta_0 \geq 0$ with $\alpha_0+\beta_0>0$. Then for
any fixed $\delta\geq -\beta_0$ function
 \begin{equation*}
     g:[1,\infty)\times[1,\infty)\rightarrow \boldsymbol{B}_+(\mathcal{H})
 \end{equation*}
 given by
 \begin{equation*}
     g(\lambda,\mu)=B^\frac{-\mu}{2}(B^\frac{\mu}{2}A^\lambda B^\frac{\mu}{2})^\frac{\delta+\beta_0\mu}{\alpha_0\lambda+\beta_0\mu}B^\frac{-\mu}{2}
 \end{equation*}
 is an increasing function of both $\lambda$ and $\mu$ for $\lambda\geq 1$ and  $\mu \geq 1 $ such, that $\alpha_0 \lambda \geq \delta$. 
 \end{theorem}
 
 Now we obtain a chain of inequalities of this type for invertible operators which satisfies    \eqref{xukl} for some $S\subset \natu$. We begin with a generalization of the \cite[Lemma 1]{mia}.
 The proof of the following lemma is analogous to the proof of mentioned Lemma. 

 \begin{lemma}\label{mia2}
 Let $A$ be an invertible operator such that
 \begin{equation*}
     (A^{*m}|A^p|^{2k}A^m)^\frac{m}{pk+m}\geq |A^m|^2,
 \end{equation*}
 for some $k\in(0;\infty)$ and  $m,p\in \natu$. Then for any fixed $\delta\geq -m$ function
 \begin{equation*}
     f_{p,\delta}:(0,\infty)\rightarrow \boldsymbol{B}_+(\mathcal{H}),
 \end{equation*}
defined by
 \begin{equation*}
     f_{p,\delta}(l)=(A^{*m}|A^p|^{2l}A^m)^\frac{\delta+m}{pl+m}
 \end{equation*}
 is increasing for $l\geq \max\{k,\frac{\delta}{p}\}$.
 \end{lemma}
 \begin{proof}
 Let $A^m=U_m|A^m|$ be  the polar decomposition of $A^m$. Since $A^m$ is invertible operator, we conclude that  $U_m$ is unitary operator.
 
 Suppose now that the following  inequality holds
 \begin{equation}\label{n31}
 (A^{*m}|A^p|^{2k}A^m)^\frac{m}{pk+m}\geq |A^m|^2.
 \end{equation}
Since $A^{*m}=U_m^*|A^{*m}|$, we get
 \begin{align*}
   (A^{*m}|A^p|^{2k}A^m)^\frac{m}{pk+m}&=
   (U_m^*|A^{*m}||A^p|^{2k}|A^{*m}|U_m)^\frac{m}{pk+m}\\&=U_m^*(|A^{*m}||A^p|^{2k}|A^{*m}|)^\frac{m}{pk+m}U_m.
 \end{align*}
 This and \eqref{n31} imply that
 \begin{equation*}
     U_m^*(|A^{*m}||A^p|^{2k}|A^{*m}|)^\frac{m}{pk+m}U_m\geq |A^m|^2.
 \end{equation*}
 We see that the above inequality is equivalent to the following one
 \begin{equation*}
     (|A^{*m}||A^p|^{2k}|A^{*m}|)^\frac{m}{pk+m}\geq U_m |A^m|^2U_m^*=|A^{*m}|^2.
 \end{equation*}
 Let $A=|A^p|^{2k}$ and $B=|A^{*m}|^2$. Then
the last inequality takes the form
 \begin{equation*}
     (B^\frac{1}{2}AB^\frac{1}{2})^\frac{m}{pk+m}\geq B.
 \end{equation*}
 Applying Theorem \ref{twc} we see that  for every real $\delta\geq -m$ funkction
 \begin{align*}
     g(\lambda)&=B^\frac{-1}{2}(B^\frac{1}{2}A^\lambda B^\frac{1}{2})^\frac{\delta+m}{pk\lambda+m}B^\frac{-1}{2}\\&=
     |A^{*m}|^{-1}(|A^{*m}||A^p|^{2k\lambda} |A^{*m}|)^\frac{\delta+m}{pk\lambda+m}|A^{*m}|^{-1}  
 \end{align*}
 is increasing for $\lambda \geq 1$ such that $pk\lambda \geq \delta$. Set $\lambda=\frac{l}{k}$, then
 \begin{align*}
     g(\frac{l}{k})&=
     |A^{*m}|^{-1}(|A^{*m}||A^p|^{2l} |A^{*m}|)^\frac{\delta+m}{pl+m}|A^{*m}|^{-1}
     \\&=|A^{*m}|^{-1}(U_mU_m^*|A^{*m}||A^p|^{2l} |A^{*m}|U_mU_m^*)^\frac{\delta+m}{pl+m}|A^{*m}|^{-1}
     \\&=|A^{*m}|^{-1}(UA^{*m}|A^p|^{2l} A^{m}U_m^*)^\frac{\delta+m}{pl+m}|A^{*m}|^{-1}
     \\&=(A^{*m})^{-1}(A^{*m}|A^p|^{2l} A^{m})^\frac{\delta+m}{pl+m}(A^{m})^{-1}
     \\&=(A^{*m})^{-1}f_{p,\delta}(l)(A^{m})^{-1}.
 \end{align*}
 This means that $f_{p,\delta}(l)$ is increasing for 
 $l\geq k$ such that $pl\geq \delta$, which
completes the proof.
 \end{proof}

 The main technical result of
this section is the following Theorem.
\begin{theorem}\label{mpj}
Let $m,n\in \natu$ be such that $n\leq m$ and $A$ be an invertible operator, which satisfy  \eqref{xukl}  with $S=\{m,n+m\}$ and $A^{*n}A^n\leq (A^{*}A)^n$. Then the following inequality hold:
 \begin{equation}\label{piekne}
 (|A^{p+m}|^{\frac{2m}{m+p}}\geq (A^{*m}|A^p|^\frac{2n}{p}A^m)^\frac{m}{m+n}\geq|A^m|^2
 \end{equation}
 for $p=n+im$, where $i\in \mathbb{Z}_+$.
\end{theorem}

\begin{proof}
We
use an induction on $p$. We easily check that both inequalities hold for $p=n$:
\begin{align*}
    (A^{*m}|A^n|^\frac{2n}{n}A^m)^\frac{m}{m+n}=(A^{*m}A^{*n}A^nA^m)^\frac{m}{m+n}=|
    A^{m+n}|^\frac{2m}{m+n}=|A|^{2m}=|A^m|^2
\end{align*}
and
\begin{equation*}
    (|A^{m+n}|^{\frac{2m}{m+n}}=|A^m|^{2}=(A^{*m}|A^n|^\frac{2n}{n}A^m)^\frac{m}{m+n}.
\end{equation*}
Suppose that both inequalities hold for a given
integer  $p$. We show that both of them hold for $p+m$ as well.
First, we prove that  the second inequality in \eqref{piekne} holds. By induction hypothesis, we have 
\begin{equation*}
    |A^{p+m}|^{\frac{2m}{m+p}}\geq|A^m|^2.
\end{equation*}
Applying the L\"owner-Heinz inequality with the power $\frac{n}{m}$ to $|A^{p+m}|^{\frac{2m}{m+p}}$ and $|A^m|^2$   and using  $A^{*s} A^{s}=(A^*A)^s$ for $s=n,m$, we get the following: \begin{equation*}
    |A^{p+m}|^{\frac{2n}{m+p}}\geq|A^n|^2.
\end{equation*}
Multiplying both sides of the above inequality on the left by $A^{*m}$ and on the right by $A^{m}$ and using  $A^{*s} A^{s}=(A^*A)^s$ for $s=n,m+n$ 
gives
\begin{equation*}
    A^{*m}|A^{p+m}|^{\frac{2n}{m+p}}A^{m}\geq A^{*m}|A^n|^2A^{m}=|A|^{2(m+n)}.
\end{equation*}
Applying L\"owner-Heinz  inequality again with the power $\frac{m}{m+n}$  and using $A^{*s} A^{s}=(A^*A)^s$ for $s=m,m+n$, we conclude that
\begin{equation}\label{pc}
    (A^{*m}|A^{p+m}|^{\frac{2n}{m+p}}A^{m})^\frac{m}{m+n}\geq |A^m|^2,
\end{equation}
which completes the induction argument for the proof of the  second  inequality  in \eqref{piekne}.

Now we turn to the proof of the first inequality in \eqref{piekne}. To
make the notation more readable, we write $p^\prime$ instead of $p+m$.
Note that the inequality in \eqref{pc} could be rewritten in the following form
\begin{equation*}
    (A^{*m}|A^{p^\prime}|^\frac{2n}{p^\prime}A^m)^\frac{m}{m+{p^\prime}\frac{n}{p^\prime}}\geq|A^m|^2. 
\end{equation*}
By  Lemma \ref{mia2} with $k=\frac{n}{p^\prime}$ function
\begin{equation*}
 f_{p^\prime,0}(l)=(A^{*m}|A^{p^\prime}|^{2l}A^m)^\frac{m}{p^\prime l+m}
\end{equation*} is increasing for $l\geq \max\{\frac{n}{p^\prime},0\}=\frac{n}{p^\prime}.$ 
In particular $f_{p^\prime,0}(1)\geq f_{p^\prime,0}(\frac{n}{p^\prime})$, which gives
\begin{equation*}
    |A^{p^\prime+m}|^{\frac{2m}{m+p^\prime}}=f_{p^\prime,0}(1)\geq f_{p^\prime,0}(\frac{n}{p^\prime})=(A^{*m}|A^{p^\prime}|^\frac{2n}{p^\prime}A^m)^\frac{m}{m+n}.  
\end{equation*}
This completes the proof.
\end{proof}
We are now in a position to formulate and prove the aforementioned analogue
of Theorem \ref{mia}.

\begin{theorem}\label{cn}
Let $m,n\in \natu$ be such that  $n\leq m$ and $A$ be an invertible which satisfies  \eqref{xukl} with $S=\{m,n+m\}$ and $A^{*n}A^n\leq (A^{*}A)^n$. Then  the following inequalities holds:
\begin{equation}\label{cnier}
 |A^{n}|^\frac{2}{n}\leq   |A^{m+n}|^\frac{2}{m+n}\leq |A^{2m+n}|^\frac{2}{2m+n}\leq \dots\leq|A^{rm+n}|^\frac{2}{rm+n} 
\end{equation}
for $r\in \natu$.
\end{theorem}
\begin{proof}
We use induction to prove that
\begin{equation}\label{cccc}
    |A^{m+p}|^\frac{2p}{m+p}\geq |A^{p}|^2,
\end{equation}
for $p=n+im$, where $i\in \mathbb{Z}_+$.
It is easy to verify that \eqref{cccc} holds for $p=n$. Suppose that above inequalitity holds for a given integer $p$. By Theorem \ref{mpj} the following inequality holds
\begin{equation*}
(A^{*m}|A^{p+m}|^\frac{2n}{p+m}A^m)^\frac{m}{m+n}\geq|A^m|^2.
\end{equation*}
Applying Lemma \ref{mia} implies
\begin{equation*}
f_{p+m,p}(l)=(A^{*m}|A^{p+m}|^{2l}A^m)^\frac{p+m}{(p+m)l+m},
\end{equation*}
is increasing for $l\geq \max \{\frac{n}{p+m},\frac{p}{p+m} \}.$
By induction hypothesis \eqref{cccc} and the monotonicity of $f_{p+m,p}$, we have
\begin{align*}
    |A^{p+m}|^2&=A^{*m}|A^{p}|^2A^m\leq A^{*m}|A^{m+p}|^\frac{2p}{p+m}A^m\\&=f_{p+m,p}(
    \frac{p}{p+m})\leq f_{p+m,p}(1)\\&=|A|^\frac{2(p+m)}{p+2m},
\end{align*}
which completes the induction argument.

Now we prove \eqref{cnier}.
Applying the L\"owner-Heinz inequality with the power $\frac{1}{p}$ 
and \eqref{cccc}, we get
\begin{equation*}
|A^{m+p}|^\frac{2}{m+p}\geq |A^{p}|^\frac{2}{p},
\end{equation*}
which completes the proof.
\end{proof}

 \bibliographystyle{amsalpha}
   
\end{document}